\documentclass[12pt,a4paper]{amsart}
\usepackage{amsmath}
\usepackage{amsthm}
\usepackage{amsfonts}
\usepackage{amssymb}
\usepackage{mathrsfs}
\usepackage{hyperref}
\usepackage{cleveref}
\usepackage[margin=1.0in]{geometry}

\numberwithin{equation}{section}

\newtheorem{proposition}{Proposition}
\newtheorem{lemma}[proposition]{Lemma}

\newtheorem{theorem}[proposition]{Theorem}
\newtheorem{corollary}[proposition]{Corollary}

\numberwithin{proposition}{section}

\theoremstyle{definition}
\newtheorem{definition}[proposition]{Definition}

\newtheorem{remark}[proposition]{Remark}

\newcommand{\rean}{\mathbb{R}}
\newcommand{\rr}[1]{\mathrm{#1}}
\newcommand{\dif}{\rr{d}}
\newcommand{\Rm}{\rr{Rm}}
\newcommand{\ric}{\rr{Ric}}
\newcommand{\vol}{\rr{Vol}}
\newcommand{\ddt}{\frac{\dif}{\dif t}}
\newcommand{\ddtt}{\frac{\dif}{\dif\tau}}

\newcommand{\pdt}{\frac{\partial}{\partial t}}
\newcommand{\leqs}{\leqslant}
\newcommand{\geqs}{\geqslant}
\newcommand{\N}{\mathcal{N}}
\newcommand{\Np}{\N^\phi}
\newcommand{\Npp}{\N^{\phi*}}
\newcommand{\Nh}{\N^H}
\newcommand{\W}{\mathcal{W}}
\newcommand{\Wp}{\W^\phi}
\newcommand{\Wh}{\W^H}
\newcommand{\hn}{\frac{n}{2}}
\newcommand{\iv}{^{-1}}

\DeclareMathOperator{\tr}{tr}
\DeclareMathOperator{\var}{Var}

\date{\today}
\author{Xilun Li}
\address{
	SMS, Peking University, Beijing 100871, China}
\email{\href{mailto:lxl28@stu.pku.edu.cn}{lxl28@stu.pku.edu.cn}}
\title{Entropy and Heat Kernel on Generalized Ricci Flow}
\begin{document}
	\maketitle
	\tableofcontents
	\begin{abstract}
		We introduce analogous geometric quantities and prove some geometric and analytic bounds in \cite{bam1} to generalized Ricci flow.
	\end{abstract}
	\section{Introduction}
	Generalized Ricci flow is a triple $(M,g_t,H_t)$, where $M$ is a smooth manifold, $g_t$, $H_t$ is a one-parameter family of metrics and closed three-forms on $M$, which satisfies the equation
	\begin{align}
		\left\{
		\begin{aligned}
			\pdt g&=-2\ric+\frac{1}{2}H^2,\\
			\pdt H&=-dd^*_gH,
		\end{aligned}
		\right.
	\end{align}
	where $H^2(X,Y):=\left<i_XH,i_YH\right>$.
	
	This flow arises in many aspects, such as generalized geometry \cite{jeffGRF}, complex geometry and mathematical physics \cite{RYM}. In complex geometry, G. Tian and J. Streets introduce the pluriclosed flow in \cite{pluriclosed} to study the geometry of non-K\"{a}hler  manifold as extensions of the K\"{a}hler Ricci flow. Let $(M^{2n},J,\omega)$ be a complex manifold, the pluriclosed flow is
	\begin{equation}
		\pdt\omega=-\rho_B^{1,1},
	\end{equation}
	where $\rho_B$ is the Ricci form of Bismut connection. In \cite{plurireg}, regarding $M$ as a real manifold, pluriclosed flow is just the generalized Ricci flow up to a smooth family of diffeomorphism induced by the Lee form $\theta=-Jd^*\omega$, where the 3-form $H$ is the torsion of Bismut connection $H=d^c\omega=i(\bar{\partial}-\partial)\omega$. Thus the closeness of $H$ is equivalent to the pluriclosed condition,
	 i.e. $\partial\bar{\partial}\omega=0$.
	
	 The main motivation of pluriclosed flow is to provide a geometric-analytic approach to classify the non-K\"{a}hler complex surfaces, especially the class \uppercase\expandafter{\romannumeral7} surfaces, as what Ricci flow does in Geometrization Conjectures and K\"{a}hler-Ricci flow does in K\"{a}hler manifolds. The pluriclosed flow preserves the pluriclosed condition. According to Gauduchon's results\cite{Gauduchon}, each complex surfaces admit a pluriclosed metric. For this reason, we should mainly concern about the flows in complex dimension 2, i.e. real dimension 4.
	
	A symmetry reduction of the generalized Ricci flow will give another interesting flow, which we called Ricci-Yang-Mills flow, see \cite{RYM}. When the manifold is the total space of a principal bundle, the phenomenon of the flow will become more clear. The Ricci-Yang-Mills flow not only is a symmetric solution of the generalized Ricci flow, but also has arisen in mathematical physics literature and as a tool to understand the geometry of principal bundles.
	
	The classical Ricci flow is just the special case where $H\equiv0$. In general, since $H^2(X,X)=|i_XH|^2\geqs0$, the generalized Ricci flow is actually a kind of super Ricci flow, i.e. 
	\begin{equation}
		\pdt g\geqs -2\ric.
	\end{equation}
	
	R. H. Bamler developed a compactness and regularity theory of super Ricci flow in \cite{bam1,bam2,bam3}, which states that any sequence of pointed super Ricci flows of bounded dimensions subsequentially converges to metric flow in a certain sense, which he called the $\mathbb{F}$-limit. Moreover, the noncollapsed $\mathbb{F}$-limits of Ricci flows is smooth away a set of codimension at least 4 in parabolic sense. Since the super Ricci flow is a quite big class of flows, it seems unlikely to get structure theory for general super Ricci flow. However, for some particular cases, such as generalized Ricci flow, it's hopeful to have similar structure results as Ricci flow. A key step in the study of pluriclosed flow is the existence conjecture, see \cite[Conjecture 5.2]{plurireg}. The structure theory, if exists, will be useful to deal with the conjecture.
	
	In this paper, we will prove some theorems on generalized Ricci flow, which are analogous to the results on Ricci flow in \cite{bam1}. The paper will be organized as follows. 
	
	In Section 2, we introduce the generalized definition of the pointed Nash and Perelman entropy, which has monotonicity along generalized Ricci flow. Then we will show that the lower bound of generalized point Nash entropy implies the weighted noncollapsed condition. 
	
	In Section 3, we will show the upper bounds of the heat kernel and its gradient. 
	
	In Section 4, we show an $\varepsilon$-regularity theorem, which roughly states that the generalized Nash entropy is close to that in Euclidean implies the geometry is close to Euclidean.
	
	In Section 5, we show some estimates on $H$, particularly in lower dimensional cases, i.e. $n\leqs4$.
	
	\textbf{Acknowledgements.}
	I am grateful to my advisor Professor Gang Tian for his helpful guidance. I thank Yanan Ye and Shengxuan Zhou for inspiring discussions.
	\section{Generalized Entropy}
	\subsection{Generalized Nash Entropy and Perelman Entropy}
	
	J. Streets\cite{jeff} introduced several concepts in generalized Ricci flow, we list them here for convenience.
	\begin{definition}[\text{\cite[Definition 2.2]{jeff}}]
		Given $(M,g_t,H_t)$ a solution to generalized Ricci flow, a one-parameter family $\phi_t$ is called the dilaton flow if
		\begin{equation}
			\square\phi=\frac{1}{6}|H|^2,
		\end{equation} 
		where $\square:=\partial_t-\Delta_{g(t)}$ is the heat operator coupled with the flow.
	\end{definition}
	%\begin{lemma}[Parabolic scaling]
	%Suppose $(g(t),H(t),\phi(t))$ solves the generalized Ricci flow, then for any $\lambda>0$, $\tilde{g}(t):=\lambda g(t_0+\lambda\iv t)$, $\tilde{H}(t):=\lambda H(t_0+\lambda\iv t), \tilde{\phi}(t):=\phi(t_0+\lambda\iv t)$, then $(\tilde{g},\tilde{H},\tilde{\phi})$ also solves the generalized Ricci flow, moreover $|\tilde{H}|^2_{\tilde{g}}=\lambda\iv|H|^2_g$.
	%\end{lemma}
	\begin{definition}[\text{\cite[Definition 2.1]{jeff}}]
		Given a smooth manifold $M$, a triple $(g,H,f)$ of a Riemannian metric, closed three-form and function, \emph{generalized scalar curvature} is
		\begin{equation}
			R^{H,f}:=R-\frac{1}{12}|H|^2+2\Delta f-|\nabla f|^2,
		\end{equation}
		\emph{generalized Ricci curvature}, or \emph{twisted Bakry-Emery curvature} is
		\begin{equation}
			\ric^{H,f}:=\ric-\frac{1}{4}H^2+\nabla^2f-\frac{1}{2}(d^*_gH+i_{\nabla f}H).
		\end{equation}
	\end{definition}
	\begin{proposition}[\text{\cite[Proposition 2.4]{jeff}}]
		Let $(M,g,H)$ be the generalized Ricci flow, $\phi_t$ dilaton flow,
		\begin{equation}\label{evolveR}
			\square R^{H,\phi}=2|\ric^{H,\phi}|^2.
		\end{equation}
	\end{proposition}
	\begin{remark}
		Note that the generalized scalar curvature is not equal to the trace of generalized Ricci curvature, so it may not be bounded from below by a constant only depends on time and dimension as Ricci flow.
	\end{remark}
	\begin{corollary}
		Let $(M,g,H)$ be the generalized Ricci flow, $\phi_t$ dilaton flow, 
		\begin{equation}
			\frac{1}{12}|H|^2+|\nabla\phi|^2\leqs R+2\Delta\phi-R^\phi_0,
		\end{equation}
		where $R^\phi_0:=\min_xR^{H,\phi}(x,0)$. Integrate the inequality, we have
		\begin{equation}\label{HL2phi}
			\int_M\left(\frac{1}{12}|H|^2+|\nabla\phi|^2\right)\dif g_t\leqs \int_M\left(R-R^\phi_0\right)\dif g_t.
		\end{equation}
	\end{corollary}
	\begin{proof}
		This follows from maximum principle applied to (\ref{evolveR}).
	\end{proof}
	\begin{definition}[\text{\cite[Definition 3.1]{jeff}}]
		Let $(M,g_t,H_t,\phi_t)$ be a solution to generalized Ricci flow. Define the \emph{conjuate heat operator}
		\begin{equation*}
			\square^*:=-\partial_t-\Delta_{g(t)}+R-\frac14|H|^2.
		\end{equation*}
		Also define the \emph{weighted conjugate heat operator}
		\begin{equation*}
			\square_\phi^*u:=e^\phi \square^*\left(u e^{-\phi}\right)=-\partial_tu-\Delta_{g(t)}u+2\left<\nabla\phi,\nabla u\right>+R^{H,\phi}u.
		\end{equation*}
		These are conjugate operators in the following sense:
		\begin{equation*}
			\begin{aligned}
				\int_{t_1}^{t_2}\int_M (\square u)v\dif g_t\dif t&=\left.\int_M uv\dif g_t\right|_{t_1}^{t_2}+\int_{t_1}^{t_2}\int_M u(\square^* v)\dif g_t\dif t,\\
				\int_{t_1}^{t_2}\int_M (\square u)ve^{-\phi}\dif g_t\dif t&=\left.\int_M uve^{-\phi}\dif g_t\right|_{t_1}^{t_2}+\int_{t_1}^{t_2}\int_M u(\square^* v)e^{-\phi}\dif g_t\dif t.
			\end{aligned}
		\end{equation*}
	\end{definition}
	\begin{definition}[\text{\cite[Definition 4.1]{jeff}}]
		Assume $v=(4\pi\tau)^{-n/2}e^{-f}e^{-\phi}$, $\dif\nu=v\dif g$, $\int_M\dif\nu=1$, define the \emph{weighted Nash entropy} by
		$$\Np[g,f,\tau]:=\int f\dif\nu-\frac{n}{2}.$$
		And define the \emph{weighted Perelman entropy} by
		$$\W^\phi[g,f,\tau]:=\int(\tau R^{f+\phi}+f-n)\dif\nu=\int[\tau(R^{\phi}+|\nabla f|^2)+f-n]\dif\nu.$$
	\end{definition}
	\begin{proposition}
		Suppose $\square^*v=0$, $\ddt\tau=-1$, $(M,g,H,\phi)$ is generalized Ricci flow, then
		$$\frac{d}{d\tau}(\tau \N^\phi)=\W^\phi,\ \ \ \frac{d}{d\tau}\W^\phi=-2\tau\int|\ric^{f+\phi}-\frac{1}{2\tau}g|^2\dif\nu+\frac{1}{3}\int|H|^2\dif\nu.$$
	\end{proposition}
	\begin{proof}Since $f=-\log v-\phi-\hn\log(4\pi\tau)$, $\square^*v=0$, then
		$$\square f=2v^{-1}\Delta v-|\nabla(f+\phi)|^2-R+\frac{1}{12}|H|^2+\frac{n}{2\tau},$$
		$$\ddt\N=\int\square f\dif\nu=-\int\left(|\nabla(f+\phi)|^2+R-\frac{1}{12}|H|^2-\frac{n}{2\tau}\right)\dif\nu.$$
		Note that integration by part, we have $\int|\nabla(f+\phi)|^2\dif\nu=\int\Delta(f+\phi)\dif\nu$, then
		$$\frac{d}{d\tau}(\tau \N)=\N-\tau\ddt\N=\int\left[\tau\left(R-\frac{1}{12}|H|^2+|\nabla(f+\phi)|^2\right)+f-n\right]\dif\nu=\W.$$
		Let $u:=e^{\phi}v$, then $\square^*_\phi u=e^\phi\square^*(ue^{-\phi})=0$. Denote by $W^{H,F}:=\tau R^{H,F}+F-n$, then by the local computation in \cite[Proposition 3.12]{jeff},
		$$
		\begin{aligned}
			\frac{d}{d\tau}\W&=\frac{d}{d\tau}\int(W-\phi)ue^{-\phi}\dif g=\int\left(\square^*_\phi(Wu)+\square\phi u-\phi\square^*_\phi u\right)e^{-\phi}\dif g\\
			&=-2\tau\int|\ric^{f+\phi}-\frac{1}{2\tau}g|^2\dif\nu+\frac{1}{3}\int|H|^2\dif\nu.
		\end{aligned}$$
	\end{proof}
	\begin{proposition}[\text{\cite[Theorem 2.1]{existheatkernel}}]
		There exists the unique heat kernel $K(x,t;y,s)\in C^\infty(\{x,y\in M, t>s\};\rean_+)$ satisfies
		\begin{align*}
			\left\{\begin{aligned}
				&\square_{x,t}K(x,t;y,s)=0,\\
				&\lim_{t\searrow s}K(x,t;y,s)=\delta_y(x).
			\end{aligned}
			\right.
		\end{align*}
		Moreover, the conjugate heat kernel is exactly the same function:
		\begin{align*}
			\left\{\begin{aligned}
				&\square^*_{y,s}K(x,t;y,s)=0,\\
				&\lim_{s\nearrow t}K(x,t;y,s)=\delta_x(y).
			\end{aligned}
			\right.
		\end{align*}
	\end{proposition}
	As in \cite[Definition 5.1]{bam1}, we can define the pointed entropy, i.e. taking the conjugate heat kernel as the function $v$ :
	\begin{definition}
		Given $(x,t)$, $\dif\nu_{x,t;s}(y):=K(x,t;y,s)\dif g_s(y)=(4\pi\tau)^{-\hn}e^{-(f+\phi)}\dif g_s$, the \emph{pointed Nash entropy} is defined by
		$$\N^\phi_{x,t}(\tau):=\Np[g_{t-\tau},f_{t-\tau},\tau],\ \ \Npp_s(x,t):=\Np_{x,t}(t-s),$$
		the \emph{pointed Perelman entropy} is defined by
		$$\Wp_{x,t}(\tau):=\Wp[g_{t-\tau},f_{t-\tau},\tau].$$
	\end{definition}
	The pointed entropy is not monotone due to the existence of the torsion $H$, we need to modify the definition of entropy to get the monotonicity. 
	\begin{definition}
		Define $$\Psi_{x,t}(\tau):=\int_{t-\tau}^t\int_M|H|^2\dif\nu_{x,t;s}\dif s,\ \ P_{x,t}(\tau):=\int_0^\tau s\iv\Psi_{x,t}(s)\dif s.$$
		Then we define the \emph{generalized pointed Nash entropy} by
		$$\Nh_{x,t}(\tau):=\Np_{x,t}(\tau)-\frac{1}{3}P_{x,t}(\tau).$$
		And define the \emph{generalized pointed Perelman entropy} by
		$$\Wh_{x,t}(\tau):=\Wp_{x,t}(\tau)-\frac{1}{3}\left(\Psi_{x,t}(\tau)+P_{x,t}(\tau)\right).$$
	\end{definition}
	\begin{remark}
		Note that $\Psi_{x,t}(0)=0$, $\Psi'_{x,t}(0)=|H|^2(x,t)$, so the function $P_{x,t}$ is well-defined.
	\end{remark}
	\begin{proposition}For the generalized pointed entropy, we have
		$$\ddtt(\tau\Nh_{x,t}(\tau))=\Wh_{x,t}(\tau),\ \ \frac{d}{d\tau}\Wh_{x,t}(\tau)=-2\tau\int_M|\ric^{f+\phi}-\frac{1}{2\tau}g|^2\dif\nu-\frac{1}{3\tau}\Psi_{x,t}(\tau)\leqs0.$$
	\end{proposition}
	\begin{proof}
		\begin{equation*}
			\ddtt(\tau\Nh_{x,t}(\tau))=\ddtt(\tau\Np_{x,t}(\tau))-\frac13\ddtt(\tau P_{x,t}(\tau))=\Wp_{x,t}(\tau)-\frac13\left(P_{x,t}(\tau)+\Psi_{x,t}(\tau)\right).
		\end{equation*}
		\begin{equation*}
			\begin{aligned}
				\ddtt\Wh_{x,t}(\tau)&=\ddtt\Wp_{x,t}(\tau)-\frac13\ddtt\left(P_{x,t}(\tau)+\Psi_{x,t}(\tau)\right)\\
				&=-2\tau\int_M|\ric^{f+\phi}-\frac{1}{2\tau}g|^2\dif\nu_{x,t;t-\tau}+\frac{1}{3}\int_M|H|^2\dif\nu-\frac13\left(\int_M|H|^2\dif\nu+\frac{1}{\tau}\Psi_{x,t}(\tau)\right)\\
				&=-2\tau\int_M|\ric^{f+\phi}-\frac{1}{2\tau}g|^2\dif\nu_{x,t;t-\tau}-\frac{1}{3\tau}\Psi_{x,t}(\tau)\leqs0.
			\end{aligned}
		\end{equation*}
	\end{proof}
	\begin{remark}
		Perelman entropy in Ricci flow becomes constant implies $g$ is a gradient shrinking soliton. Our generalized pointed Perelman entropy becomes constant implies the flow is actually Ricci flow and $g$ is also gradient shrinking soliton. It's natural since by \cite[Proposition 4.28]{jeffGRF}, the generalized shrinking soliton is just the shrinking soliton with $H\equiv0$.
	\end{remark}
	\begin{corollary}We have
		$$\Nh_{x,t}(\tau)=\frac{1}{\tau}\int^\tau_0\Wh_{x,t}(s)\dif s\geqs \Wh_{x,t}(\tau),\ \ \ddtt\Nh_{x,t}(\tau)\leqs0 ,$$
		$$\int_M\tau(R^{\phi}+|\nabla f|^2)\dif\nu_{t-\tau}\leqs\frac{n}{2}+\frac{1}{3}\Psi_{x,t}(\tau).$$
	\end{corollary}
	\begin{proof}The first follows from $\tau\ddtt\Nh=\Wh-\Nh\leqs0$.
		The second conclusion follows from $\Wh=\tau\int(R^\phi+|\nabla f|^2)\dif\nu+\Nh-\hn-\frac{1}{3}\Psi_{x,t}(\tau)$.
	\end{proof}
	We need an oscillation estimation of Nash entropy analogous to \cite[Corollary 5.11]{bam1}, which can be divided by two parts to estimate. The first term, weighted entropy, can be dealt as \cite[Theorem 5.9]{bam1}.
	\begin{proposition}\label{heatnash}
		Suppose $R^\phi(\cdot,s)\geqs R^\phi_{\min}$, then
		$$|\nabla \Npp_s|\leqs\left(\frac{n}{2(t-s)}-R^\phi_{\min}+\frac{1}{3(t-s)}\Psi_{x,t}(t-s)\right)^{\frac{1}{2}},\ \ \ -\frac{n}{2(t-s)}\leqs\square \Npp_s\leqs0.$$
	\end{proposition}
	\begin{proof}
		The proof follows the argument in \cite[Proof of Theorem 5.9]{bam1} with a modification. We assume $s=0$ after application of a time-shift.
		$$\Npp_0(x,t)=-\int K(x,t;y,0)(\log K(x,t;y,0)+\phi(y,0))\dif g_0(y)-\frac{n}{2}\log(4\pi t)-\frac{n}{2}.$$
		For any vector $v\in T_xM$ with $|v|_t=1$, 
		$$\begin{aligned}
			\partial_v\Npp_0(x,t)&=-\int\partial_vK(x,t;y,0)(\log K(x,t;y,0)+1+\phi(y,0))\dif g_0(y)\\
			&=\int\partial_vK(x,t;y,0)\left(f-\frac{n}{2}-\Npp_0(x,t)\right)\dif g_0(y)\\
			&=\int\frac{\partial_vK(x,t;y,0)}{K(x,t;y,0)}\left(f-\frac{n}{2}-\Npp_0(x,t)\right)\dif\nu_{x,t;0}(y).
		\end{aligned}
		$$
		$$\int\left(f-\frac{n}{2}-\Npp_0(x,t)\right)^2\dif\nu(y)\leqs 2t\int|\nabla f|^2\dif\nu\leqs n-2tR^\phi_{\min}+\frac{2}{3}\Psi_{x,t}(t).$$
		$$|\partial_v\Npp_0|^2(x,t)\leqs\frac{1}{2t}\left(n-2tR^\phi_{\min}+\frac{2}{3}\Psi_{x,t}(t)\right).$$
		For the last two inequality, we use the $L^2$-Poincar\'{e} inequality in \cite[Theorem 1.10]{HN14} and the integral bounds on the gradient of the heat kernel in \cite[Proposition 4.2]{bam1}. Both of them hold for super Ricci flow.
		
		The proof of the second result is same as that in \cite{bam1}.
		\begin{equation*}
			\begin{aligned}
				\square\Npp_0(x,t)&=-\int_M\square_{x,t}\left(K(x,t;y,0)(\log K(x,t;y,0)+\phi(y,0))\right)\dif g_0(y)-\frac{n}{2t}\\
				&=-\int_M\left(\square_{x,t}K(\log K+\phi(y,0)+1)-\frac{|\nabla_xK(x,t;y,0)|^2}{K(x,t;y,0)}\right)\dif g_0(y)-\frac{n}{2t}\\
				&=\int_M\frac{|\nabla_xK(x,t;y,0)|^2}{K(x,t;y,0)}\dif g_0(y)-\frac{n}{2t}.
			\end{aligned}
		\end{equation*} 
		Then the second results also follow from the integral bounds of the heat kernel, see \cite[Proposition 4.2]{bam1}.
	\end{proof}
	\begin{definition}
		Given two probability measure $\mu_1,\mu_2$ on $(M,g)$, define the \emph{$W_1$-Wasserstein distance} by
		$$d_{W_1}^g(\mu_1,\mu_2):=\sup\left\{\int_Mf\dif(\mu_1-\mu_2):f\in C^{\infty}(M),|\nabla f|\leqs 1\right\}.$$
	\end{definition}
	\begin{lemma}\label{w1monotone}
		Let $(M,g(t),H(t))$ be a generalized Ricci flow, $x_1,x_2\in M$, then
		\begin{align*}
			t\mapsto d_{W_1}^{g(t)}(\nu_{x_1,t_1}(t),\nu_{x_2,t_2}(t))
		\end{align*}
		is non-decreasing for $t\leqs t_1,t_2$.
	\end{lemma}
	\begin{proof}
		Note that the generalized Ricci flow is a super Ricci flow, then it follows from \cite[Lemma 2.7]{bam1}.
	\end{proof}
	\begin{corollary}\label{nashosc}
		Suppose $R^\phi(\cdot,s)\geqs R^\phi_{\min}$, $|H|^2\leqs A$ on $t\in[s,t^*]$, $s<t^*\leqs t_1,t_2$, then for $x_1,x_2\in 
		M$
		$$\Npp_s(x_1,t_1)-\Npp_s(x_2,t_2)\leqs\left(\frac{n}{2(t^*-s)}-R^\phi_{\min}+\frac{1}{3}A\right)^{\frac{1}{2}}d_{W_1}^{t^*}\left(\nu_{x_1,t_1}(t^*),\nu_{x_2,t_2}(t^*)\right)+\frac{n}{2}\log\left(\frac{t_2-s}{t^*-s}\right).$$
	\end{corollary}
	\begin{proof}
		It follows from \Cref{heatnash} with same argument as \cite[Proof of Corollary 5.11]{bam1}.
	\end{proof}
	We can get a gradient bound on $P^*_s$ by direct computation.
	\begin{proposition}\label{pf2.19}
		Suppose $|H|^2\leqs A$, $P^*_s(x,t):=P_{x,t}(t-s)$, we have $|\nabla P^*_{s}(x,t)|\leqs C_nA\sqrt{t-s}$.
	\end{proposition}
	\begin{proof}By changing the order of integration, we have
		$$
		\begin{aligned}
			P^*_{t-\tau}(x,t)&=\int_0^\tau s\iv\Psi_{x,t}(s)\dif s=\int_0^{\tau}\int_0^s\int_Ms\iv|H|^2\dif\nu_{t-r}\dif r\dif s\\
			&=\int_0^\tau\int_r^\tau\int_Ms\iv|H|^2\dif\nu_{t-r}\dif s\dif r=\int_0^\tau\log\left(\frac{\tau}{s}\right)\int_M|H|^2\dif\nu_{t-s}\dif s.
		\end{aligned}$$
		$$
		\begin{aligned}
			\left|\nabla P^*_{t-\tau}\right|&\leqs\int_0^\tau\log\left(\frac{\tau}{s}\right)\int_M|H|^2\frac{\left|\nabla K(x,t;y,t-s)\right|}{K(x,t;y,t-s)}\dif\nu_{t-s}(y)\dif s\\
			&\leqs C_nA\int_0^\tau s^{-\frac{1}{2}}\log\left(\frac{\tau}{s}\right)\dif s=C_nA\sqrt{\tau}.
		\end{aligned}$$
		where we use integral bounds on the gradient of heat kernel along super Ricci flow in \cite[Proposition 4.2]{bam1}.
	\end{proof}
	The oscillation bound of $P^*_s$ follows from a straightforward computation:
	\begin{proposition}\label{Posc}
		Suppose $0<t^*<t_1<t_2$, $|H|^2+|\nabla H|^2\leqs C$, $d_{W_1}^{t^*}(\nu_{x_1,t_1}(t^*),\nu_{x_2,t_2}(t^*))\leqs \delta$, then
		$$P^*_0(x_1,t_1)\leqs P^*_0(x_2,t_2)+C\left(t_1-(1-\delta)\int_0^{t^*}\log\left(\frac{t_2}{t_2-s}\right)\dif s\right).$$
		In particular, assume $t_2-t^*<\varepsilon<10\iv$, $0<T\iv<t_2<T$, then
		$$P^*_0(x_1,t_1)\leqs P^*_0(x_2,t_2)+C(\varepsilon T-\varepsilon\log\varepsilon+2\delta T).$$
	\end{proposition}
	\begin{proof}As the computation in Proof of Proposition \ref{pf2.19}, we have
		\begin{align*}
			P^*_0(x_i,t_i)=\int_0^{t_i}\log\left(\frac{t_i}{t_i-s}\right)\int_M|H|^2\dif\nu_{x_i,t_i}(s)\dif s=\int_0^{t^*}+\int_{t^*}^{t_i},
		\end{align*}
		where $i=1,2$.
		For the second term, we have
		$$0\leqs\int_{t^*}^{t_i}\log\left(\frac{t_i}{t_i-s}\right)\int_M|H|^2\dif\nu_{x_i,t_i}(s)\dif s\leqs C\int_{t^*}^{t_i}\log\left(\frac{t_i}{t_i-s}\right)\dif s,$$
		where we use $|H|^2\leqslant C$. Then we have
		\begin{align*}
			P^*_0(x_1,t_1)-P^*_0(x_2,t_2)\leqslant& \int_0^{t^*}\log\left(\frac{t_1}{t_1-s}\right)\int_M|H|^2\dif\nu_{x_1,t_1}(s)\dif s+C\int_{t^*}^{t_1}\log\left(\frac{t_1}{t_1-s}\right)\dif s\\
			&-\int_0^{t^*}\log\left(\frac{t_2}{t_2-s}\right)\int_M|H|^2\dif\nu_{x_2,t_2}(s)\dif s.
		\end{align*}
		By lemma \ref{w1monotone}, we have $d_{W_1}^{t}(\nu_{x_1,t_1}(t),\nu_{x_2,t_2}(t))\leqs \delta$ for any $t\leqs t^*$.
		$$
		\begin{aligned}
			&\int_0^{t^*}\log\left(\frac{t_1}{t_1-s}\right)\int_M|H|^2\dif\nu_{x_1,t_1}(s)\dif s-\int_0^{t^*}\log\left(\frac{t_2}{t_2-s}\right)\int_M|H|^2\dif\nu_{x_2,t_2}(s)\dif s\\
			=&\ I_1+I_2,
		\end{aligned} $$
		where
		\begin{align*}
			I_1&:=\int_0^{t^*}\left[\log\left(\frac{t_1}{t_1-s}\right)-\log\left(\frac{t_2}{t_2-s}\right)\right]\int_M|H|^2\dif\nu_{x_1,t_1}(s)\dif s\\
			&\leqs C\int_0^{t^*}\left[\log\left(\frac{t_1}{t_1-s}\right)-\log\left(\frac{t_2}{t_2-s}\right)\right]\dif s,\\
			I_2&:=\int_0^{t^*}\log\left(\frac{t_2}{t_2-s}\right)\int_M|H|^2\dif\left(\nu_{x_1,t_1}(s)-\nu_{x_2,t_2}(s)\right)\dif s\\
			&\leqs C\delta\int_0^{t^*}\log\left(\frac{t_2}{t_2-s}\right)\dif s.
		\end{align*}
		The inequality on $I_1$ is due to $|H|^2\leqslant C$ and the fact that $\log\left(\frac{t_1}{t_1-s}\right)-\log\left(\frac{t_2}{t_2-s}\right)\geqslant0$. The inequality on $I_2$ is due to the definition of $d_{W_1}$ and $\left|\nabla|H|^2\right|\leqslant 2|H||\nabla H|\leqslant C$.
		
		Combining above, we have
		\begin{align*}
			P^*_0(x_1,t_1)-P^*_0(x_2,t_2)\leqs&C\left\{\int_0^{t^*}\left[\log\left(\frac{t_1}{t_1-s}\right)-\log\left(\frac{t_2}{t_2-s}\right)\right]\dif s+\delta\int_0^{t^*}\log\left(\frac{t_2}{t_2-s}\right)\dif s\right.\\
			&\left.+\int_{t^*}^{t_1}\log\left(\frac{t_1}{t_1-s}\right)\dif s\right\}\\
			 \leqs&C\left\{\int_0^{t_1}\log\left(\frac{t_1}{t_1-s}\right)\dif s-\int_0^{t^*}\log\left(\frac{t_2}{t_2-s}\right)\dif s+\delta\int_0^{t^*}\log\left(\frac{t_2}{t_2-s}\right)\dif s\right\}\\
			\leqslant&C\left(t_1-(1-\delta)\int_0^{t^*}\log\left(\frac{t_2}{t_2-s}\right)\dif s\right).
		\end{align*}
		If $t_2-t^*<\varepsilon<10\iv$, $0<T\iv<t_2<T$, then
		\begin{align*}
			\int_0^{t^*}\log\left(\frac{t_2}{t_2-s}\right)\dif s=t^*\log t_2-t_2\log t_2+t^*+(t_2-t^*)\log(t_2-t^*)\leqs T+\log T\leqs 2T,
		\end{align*}
		\begin{align*}
			t_1-\int_0^{t^*}\log\left(\frac{t_2}{t_2-s}\right)\dif s=&(t_1-t^*)+(t_2-t^*)\log t_2-(t_2-t^*)\log(t_2-t^*)\\
			\leqslant& \varepsilon+\varepsilon \log T-\varepsilon\log\varepsilon\\
			\leqslant& \varepsilon T-\varepsilon\log\varepsilon.
		\end{align*}
		So
		\begin{align*}
			P^*_0(x_1,t_1)-P^*_0(x_2,t_2)\leqs C(\varepsilon T-\varepsilon\log\varepsilon+2\delta T).
		\end{align*}
	\end{proof}
	\subsection{Weighted noncollapsed}
	\begin{definition}
		Define the \emph{$\phi$-weighted volume} by
		$$|B(x,t,r)|_\phi:=\int_{B(x,t,r)}e^{-\phi(y,t)}\dif g_t(y).$$
	\end{definition}
	\begin{remark}
		Since $\square\phi=\frac16|H|^2\geqslant0$, if we assume $\phi$ has nonnegative initial data, then $\phi(x,t)$ will always be nonnegative by maximum principle. Then we have $|B(x,t,r)|_\phi\leqs|B(x,t,r)|_{g_t}$. Thus weighted noncollapsed condition is stronger than the conventional condition.
	\end{remark}
	\begin{theorem}\label{noncollaped}
		Suppose $[t_0-r^2,t_0]\subset I$, $|R^\phi|+|H|^2\leqs Cr^{-2}$ on $B(x_0,t_0,r)\times[t_0-r^2,t_0]$, 
		$\phi\geqs0$, then
		$$|B(x_0,t_0,r)|_\phi\geqs c\exp(\N^H_{x_0,t_0}(r^2))r^n.$$
	\end{theorem}
	\begin{proof}The proof follows the argument in \cite[Proof of Theorem 6.1]{bam1} with a  modification.
		After parabolic scaling, we can assume $r=1$, $t_0=0$. 
		
		Take $h(y):=\eta(d_0(x_0,y))\in C^\infty_c(B(x_0,0,1))$, where $\eta$ is cut-off function, such that $0\leqs\eta\leqs1$, $\eta|_{[0,\frac{1}{2}]}\equiv 1$, $\eta|_{[1,\infty)}\equiv0$, $0\geqs\eta'\geqs-4$. $a:=\int_M h^2e^{-\phi}\dif g_0$, then $|B(x_0,0,\frac{1}{2})|_\phi\leqs a\leqs |B(x_0,0,1)|_\phi$.
		
		Take $v\in C^\infty(M\times[-1,0])$, such that $\square^*_\phi v=0$ and $v(\cdot,0)=a\iv h^2$, then $\int ve^{-\phi}\dif g=\int a\iv h^2e^{-\phi}\dif g_0=1$. $v_t=(4\pi\tau)^{-\hn}e^{-f_t}$, where $\tau:=1-t$. $f_{-1}=-\log v_{-1}-\hn\log(8\pi)$, $f_0=-\log (a\iv h^2)-\hn\log(4\pi)$.
		$$\Np[g_{-1},f_{-1},2]=\int f_{-1}v_{-1}e^{-\phi}\dif g_{-1}-\hn=-\int v_{-1}\log v_{-1}e^{-\phi}\dif g_{-1}-C_n.$$
		Note that $v_{-1}(z)e^{-\phi(z,-1)}=\int K(y,0;z,-1)a\iv h^2(y)e^{-\phi(y,0)}\dif g_0(y)$, then by Jensen inequality,
		$$v_{-1}\log v_{-1}(z)\leqs\int e^{\phi(z,-1)}K(y,0;z,-1)\left(\log K(y,0;z,-1)+\phi(z,-1)\right)a\iv h^2(y)e^{-\phi(y,0)}\dif g_0(y).$$
		Note that $\Np_{y,0}(1)=-\int K(y,0;z,-1)\left(\log K(y,0;z,-1)+\phi(z,-1)\right)\dif g_{-1}(z)+C_n$, then
		$$\Np[g_{-1},f_{-1},2]\geqs\int\Np_{y,0}(1)a\iv h^2(y)e^{-\phi(y,0)}\dif g_0(y)-C_n\geqs\Np_{x,0}(1)-C\geqs \Nh_{x,0}(1)-C.$$
		$$
		\begin{aligned}
			\int_0^1\Wp[g_{-s},f_{-s},1+s]\dif s&=\int_0^1\frac{\dif}{\dif s}\left((1+s)\Np[g_{-s},f_{-s},1+s]\right)\dif s\\
			&=2\Np[g_{-1},f_{-1},2]-\Np[g_0,f_0,1]\\
			&\geqs2\Np_{x,0}(1)-\Np[g_0,f_0,1]-C,
		\end{aligned}$$
		$$\int_0^1\Wp[g_{-s},f_{-s},1+s]\dif s\leqs\Wp[g_0,f_0,1]+C,$$
		$$
		\begin{aligned}
			\Wp[g_0,f_0,1]+\Np[g_0,f_0,1]&=\int\left(R^\phi+|\nabla f_0|^2+2f_0\right)\dif\nu_0-C_n\\
			&\leqs a\iv\int(4|\nabla h|^2-2h^2\log h^2)e^{-\phi}\dif g_0+2\log a+C\\
			&\leqs a\iv(100+2e\iv)|B(x_0,0,1)|_\phi+2\log|B(x_0,0,1)|_\phi+C\\
			&\leqs200\frac{|B(x_0,0,1)|_\phi}{|B(x_0,0,\frac{1}{2})|_\phi}+2\log|B(x_0,0,1)|_\phi+C.
		\end{aligned}$$
		Then by parabolic rescale, we have
		$$r^{-n}|B(x_0,0,r)|_\phi\geqs c\exp(\Nh_{x,0}(r^2))\exp\left(-100\frac{|B(x_0,0,r)|_\phi}{|B(x_0,0,\frac{r}{2})|_\phi}\right).$$
		Then there exists $c_0=c_0(c,n)$, we claim that $r^{-n}|B(x_0,0,r)|_\phi\geqs c_0\exp(\Nh_{x,0}(r^2))$. If not, there exists $r_0>0$, $r_0^{-n}|B(x_0,0,r_0)|_\phi< c_0\exp(\Nh_{x,0}(r_0^2))$, then
		\begin{align*}
			\frac{|B(x_0,0,r_0)|_\phi}{|B(x_0,0,r_0/2)|_\phi}>100\iv\log(\frac{c}{c_0})>2^n.
		\end{align*}
		  so $r_0/2$ is also a counterexample for the claim. Repeating this process, we get for any $k$, $r_k=2^{-k}r_0$ violates the claim, however it's impossible if we set $c_0<\omega_n$ since as $r\to 0$, $r^{-n}|B(x_0,0,r)|_\phi\to \omega_ne^{-\phi(x_0,0)}$, $\exp(\Nh_{x,0}(r_0^2))\to\exp(-\phi(x_0,0)-\frac{1}{3}|H|^2(x_0,0))$.
	\end{proof}
	\begin{remark}
		Note that the key point in the proof is the monotonicity of the generalized entropy $\Nh_{x,t}$. For the weighted entropy $\Np_{x,t}$, the last argument of the proof fails.
	\end{remark}
	Bamler introduced the concept of $H_n$-center instead of worldline to represent the same point in different time with error.
	\begin{definition}[\text{\cite[Definition 3.10]{bam1}}]
		A point $(z,t)\in M\times I$ is called an $H_n$-center of a point $(x_0,t_0)\in M\times I$ if $t\leqs t_0$ and $\var_t(\delta_z,\nu_{x_0,t_0;t})\leqs H_n(t_0-t)$,
		where the variance $\var(\mu_1,\mu_2):=\int_{M\times M}d^2(x_1,x_2)\dif\mu_1(x_1)\dif\mu_2(x_2)$.
	\end{definition} 
	In \cite[Proposition 3.12]{bam1}, Bamler shows that along the super Ricci flow, $H_n$-center always exists due to the $H_n$-concentration of such flow. Considering the $H_n$-center instead, we can have a similar result using the weaker entropy $\Np_{x,t}$.
	\begin{theorem}\label{centernoncollapsed}
		Suppose $[t_0-r^2,t_0]\subset I$, $R^\phi(\cdot,t_0-r^2)\geqs R^\phi_{\min}$, $|H|^2\leqs Ar^{-2}$, $(z,t_0-r^2)$ is $H_n$-center of $(x_0,t_0)$, then
		$$|B(z,t_0-r^2,\sqrt{2H_n}r)|_\phi\geqs c\exp(\Np_{x_0,t_0}(r^2))r^n,$$
		where $c=c\exp(-2(n-2R^\phi_{\min}r^2+\frac{2}{3}A)^{\frac{1}{2}})$.
	\end{theorem}
	\begin{proof}The proof is analogous to the argument in \cite[Proof of Theorem 6.2]{bam1}.
		After parabolic scaling, we can assume $r=t_0=1$. Denote by $\dif\nu_t:=\dif\nu_{x_0,1;t}$, $B:=B(z,0,\sqrt{2H_n})$. By \cite[Proposition 3.13]{bam1}, we have $\nu_0(B)\geqs\frac{1}{2}$.
		$$\int_M\left|f-\hn-\Np_{x_0,1}(1)\right|\dif\nu_0\leqs\left(\int_M\left(f-\hn-\Np_{x_0,1}(1)\right)^2\dif\nu_0\right)^{\frac{1}{2}}\leqs\left(n-2R^\phi_{\min}+\frac{2}{3}\Psi_{x_0,1}(1)\right)^{\frac{1}{2}}.$$
		$$
		\begin{aligned}
			\frac{1}{\nu_0(B)}\int_Bf\dif\nu_0&\geqs\hn+\Np_{x_0,1}(1)-\frac{1}{\nu_0(B)}\int_M\left|f-\hn-\Np_{x_0,1}(1)\right|\dif\nu_0\\
			&\geqs\hn+\Np_{x_0,1}(1)-2\left(n-2R^\phi_{\min}+\frac{2}{3}\Psi_{x_0,1}(1)\right)^{\frac{1}{2}}.
		\end{aligned}$$
		Let $u:=\nu_0(B)\iv(4\pi)^{-\hn}e^{-f}$, then $\int_Bue^{-\phi}\dif g_0=1$.
		$$
		\begin{aligned}
			\int_Bu\log ue^{-\phi}\dif g_0&=-\frac{1}{\nu_0(B)}\int f\dif\nu_0-\hn\log(4\pi)-\log(\nu_0(B))\\
			&\leqs-\Np_{x_0,1}(1)-\hn+2\left(n-2R^\phi_{\min}+\frac{2}{3}\Psi_{x_0,1}(1)\right)^{\frac{1}{2}}+\log2+\hn\log(4\pi).
		\end{aligned}$$
		By Jensen's inequality,
		$$\frac{1}{|B|_\phi}\int_Bu\log ue^{-\phi}\dif g_0\geqs\left(\frac{1}{|B|_\phi}\int_Bue^{-\phi}\dif g_0\right)\log\left(\frac{1}{|B|_\phi}\int_Bue^{-\phi}\dif g_0\right)=-\frac{1}{|B|_\phi}\log|B|_\phi,$$
		$$\log|B|_\phi\geqs\Np_{x_0,1}(1)-2\left(n-2R^\phi_{\min}+\frac{2}{3}\Psi_{x_0,1}(1)\right)^{\frac{1}{2}}-C_n.$$
	\end{proof}
	\section{Heat kernel Estimate}
	\begin{theorem}\label{heatkernel}
		Suppose $[s,t]\subset I$, $R^\phi(\cdot,s)\geqs R^\phi_{\min}$, $|H|^2\leqs A$, then
		$$K(x,t;y,s)\leqs\frac{C}{(t-s)^{\frac{n}{2}}}\exp(-\Npp_s(x,t)),$$
		where $C=C(R^\phi_{\min}(t-s),A(t-s))$.
	\end{theorem}
	\begin{proof}The proof follows from the argument in \cite[Proof of Theorem 7.1]{bam1}.
		After parabolic scaling, we can assume $s=0$, $t=1$. By \cite[Theorem 26.25]{CCG10}, we always have $K(x,t;y,0)\leqs Zt^{-\hn}\exp(-\Npp_0(x,t))$ for $t\in (0,1]$, for some large $Z$, however, $Z$ may depends on the flow $g(t)$. We will show that there exists $\underline{Z}=\underline{Z}(R^\phi_{\min}(t-s),A(t-s))$, such that if $Z\geqs\underline{Z}$, the bound will still hold after replacing $Z$ by $\frac{Z}{2}$. Using this argument, we have $K(x,t;y,0)\leqs \underline{Z}t^{-\hn}\exp(-\Npp_0(x,t))$. Note that the bound is parabolic scaling invariant, it suffices to show the bound holds for $t=1$. Denote by $u(x,t):=K(x,t;y,0)$, $\square u=0$.
		
		$$\ddt\int_M ue^{-\phi}\dif g_t=-\int_Mu(\square^*_\phi1)e^{-\phi}\dif g_t=-\int_MR^\phi ue^{-\phi}\dif g_t\leqs-R^\phi_{\min}\int_Mue^{-\phi}\dif g_t.$$
		Then we have $\int_Mue^{-\phi}\dif g_t\leqs C\exp(-\phi(y,0))$. Denote by $v:=(t-\frac12)|\nabla u|^2+u^2$.
		$$\square v\leqs-2(t-\frac12)|\nabla^2u|^2-2|\nabla u|^2.$$
		Then $\square v\leqs0$ for $t\in[\frac12,1]$. We assume $t\geqs\frac12$ from now on, we have
		$$
		\begin{aligned}
			v(x,t)&\leqs\int_MK(x,t;y,\frac12)v(y,\frac12)\dif g_{\frac12}(y)=\int_MK(x,t;y,\frac12)u^2(y,\frac12)\dif g_{\frac12}(y)\\
			&\leqs2^nZ^2\int_MK(x,t;y,\frac12)\exp(-2\Npp_0(y,\frac12))\dif g_{\frac12}(y).
		\end{aligned}
		$$
		Take $(z,\frac12)$ be the $H_n$-center of $(x,t)$, we have $d_{W_1}^{\frac12}(\delta_z,\nu_{x,t}(\frac12))\leqs\sqrt{\frac12H_n}$. We have $-\Npp_0(y,\frac12)\leqs-\Npp_0(z,\frac12)+Cd_{\frac12}(y,z)\leqs-\Npp_0(x,t)+Cd_{\frac12}(y,z)+C$.
		$$
		\begin{aligned}
			v(x,t)&\leqs2^nZ^2\int_MK(x,t;y,\frac12)\exp(-2\Npp_0(y,\frac12))\dif g_{\frac12}(y)\\
			&\leqs CZ^2\exp(-2\Npp_0(x,t))\int_MK(x,t;y,\frac12)\exp(Cd_{\frac12}(y,z))\dif g_{\frac12}(y).
		\end{aligned}$$
		By \cite[Proof of Theorem 7.1, Theorem 3.14]{bam1}, we have
		$$\int_MK(x,t;y,\frac12)\exp(Cd_{\frac12}(y,z))\dif g_{\frac12}(y)\leqs C.$$
		Note that $v(x,t)\geqs(t-\frac12)|\nabla u|^2$, we have for $t\in[\frac34,1]$, $|\nabla u|(x,t)\leqs CZ\exp(-\Npp_0(x,t))$.
		
		Fix $(x_0,1)$, denote by $\dif\nu:=K(x_0,1;\cdot,\cdot)\dif g$, let $\rho\in(0,\frac12)$ to be determined, then $t_1:=1-\rho^2\in [\frac34,1]$. Let $(z_1,t_1)$ be the $H_n$-center of $(x_0,1)$.
		$$-\Npp_0(z_1,t_1)\leqs-\Npp_0(x_0,1)+Cd_{W_1}^{t_1}(\delta_{z_1},\nu_{t_1})\leqs-\Npp_0(x_0,1)+C\sqrt{H_n\rho^2}.$$
		We have for $\rho\leqs\bar{\rho}$, $-\Npp_0(z_1,t_1)\leqs-\Npp_0(x_0,1)+\log2$. Then $-\Npp_0(y,t_1)\leqs-\Npp_0(x_0,1)+\log2+Cd_{t_1}(y,z_1)$. Denote by $B:=B(z_1,t_1,\sqrt{100H_n}\rho)$.
		$$u(x_0,1)=\int_Mu\dif\nu_{t_1}=\int_B+\int_{M\setminus B}u\dif\nu_{t_1}.$$
		By \Cref{centernoncollapsed}, $|B|_{\phi}\geqs c\exp(\Npp_0(x_0,1))\rho^n$. For any $x',x''\in B$, we have $u(x',t_1)\leqs u(x'',t_1)+CZ\exp(-\Npp_0(x,t))\rho$. Integrate by $|B|_\phi\iv e^{-\phi}\dif g_{t_1}$,
		$$
		\begin{aligned}
			u(x',t_1)&\leqs\frac{1}{|B|_\phi}\int_Bue^{-\phi}\dif g_{t_1}+CZ\exp(-\Npp_0(x,t))\rho\\
			&\leqs C\exp(-\Npp_0(x,t))(Z\rho+C\rho^{-n}).
		\end{aligned}$$
		Then for the first term, we have $\int_Bu\dif\nu_{t_1}\leqs C\exp(-\Npp_0(x,t))(Z\rho+C\rho^{-n})$. For the second term, $\nu_{t_1}(M\setminus B)\leqs 100\iv$, $\int_{M\setminus B}u\dif\nu_{t_1}\leqs Zt_1^{-\hn}\int_{M\setminus B}\exp(-\Npp_0(y,t_1))\dif\nu_{t_1}(y)$. For $\rho\leqs\bar{\rho}$, since $-\Npp_0(z_1,t_1)\leqs-\Npp_0(x_0,1)+\log2$, we have
		$$\begin{aligned}
			\int_{M\setminus B}u\dif\nu_{t_1}&\leqs 2Z\int_{M\setminus B}\exp(-\Npp_0(y,t_1))\dif\nu_{t_1}(y)\\
			&\leqs 4Z\exp(-\Npp_0(x_0,1))\int_{M\setminus B}\exp(Cd_{t_1}(y,z_1))\dif\nu_{t_1}(y).
		\end{aligned}$$
		By \cite[(7.18)]{bam1}, we have
		$$\int_{M\setminus B}\exp(Cd_{t_1}(y,z_1))\dif\nu_{t_1}(y)\leqs\nu(M\setminus B)+C\rho\leqs\frac{1}{100}+C\rho.$$
		Then we have $\int_{M\setminus B}u\dif\nu_{t_1}\leqs\exp(-\Npp_0(x,t))(\frac{1}{10}Z+CZ\rho)$.
		$$u(x_0,1)\leqs \exp(-\Npp_0(x,t))(CZ\rho+C\rho^{-n}+\frac{1}{10}Z).$$
		First fix $\rho=\rho_0$ such that $C\rho_0\leqs\frac{1}{10}$, then for $Z\geqs\underline{Z}=5C\rho_0^{-n}$, we have $$u(x_0,1)\leqs\frac12Z\exp(-\Npp_0(x,t)).$$.
	\end{proof}
	\begin{theorem}
		Suppose $[s,t]\subset I$, $R^\phi(\cdot,s)\geqs R^\phi_{\min}$, $|H|^2\leqs A$, then there exist constants $C, C_0$ depending on $R^\phi_{\min}(t-s)$ and $A(t-s)$, such that
		$$\frac{|\nabla_xK|}{K}(x,t;y,s)\leqs\frac{C}{(t-s)^{\frac{1}{2}}}\sqrt{\log\left(\frac{C_0\exp(-\Npp_s(x,t))}{(t-s)^{\frac{n}{2}}K(x,t;y,s)}\right)}.$$
	\end{theorem}
	\begin{proof}
		By \cite[Proof of Theorem 7.5]{bam1}, \Cref{heatkernel} and \Cref{nashosc} imply this theorem, since other argument holds for super Ricci flow.
	\end{proof}
	\section{$\varepsilon$-Regularity}
	We need some notions to state the theorem.
	\begin{definition}
		\emph{Backward parabolic curvature radius} is defined by
		$$r_{\Rm}(x,t):=\sup\{r>0:|\Rm|\leqs r^{-2} \text{on\ } B(x,t,r)\times [t-r^2,t]\}.$$
	\end{definition}
	Bamler introduce a new notion of parabolic neighborhood instead of conventional neighborhood in \cite{bam1}, which is an important concept in \cite{bam1,bam2,bam3}.
	\begin{definition}[\cite{bam1}]
		Suppose that $(x_0,t_0)\in M\times I$, $r>0$ and $t_0-r^2\in I$. The \emph{backward $P^*$-parabolic $r$-balls} is defined by 
		$$P^*_-(x_0,t_0,r):=\{(x,t)\in M\times [t_0-r^2,t_0]:d_{W_1}^{t_0-r^2}(\nu_{x_0,t_0;t_0-r^2},\nu_{x,t;t_0-r^2})<r\}.$$
	\end{definition}
	\begin{theorem}\label{reg}
		Suppose $[t-r^2,t]\subset I$, $r^2|H|^2+r^4|\nabla H|^2\leqs C$, $R^\phi\geqs -Cr^{-2}$, there exists $\varepsilon(n,C)>0$, if $\Nh_{x,t}(r^2)\geqs-\varepsilon$, then $r_{\Rm}(x,t)\geqs \varepsilon r$.
	\end{theorem}
	\begin{proof}The proof follows from the argument in \cite[Proof of Theorem 10.2]{bam1}.
		After a parabolic scaling, we can assume $t=r=1$. Assume there exists a sequence of counterexamples $(M_i,(g_{i,t})_{t\in[0,1]},(H_{i,t}),(\phi_{i,t}),(x_i,1))$, $r_i:=r_{\Rm}(x_i,1)<\varepsilon_i\to0$, $\Nh_{x_i,1}(1)\geqs-\varepsilon_i$.
		
		 Choose a sequence $A_i\to\infty$ with $10A_ir_i\to 0$. By a point-picking argument in \cite[Claim 10.5]{bam1}, there exists $(x_i',t_i')\in P^*_-(x_i,1,10A_ir_i)$ satisfies $r_i':=r_{\Rm}(x_i',t_i')\leqs r_i\to 0$ and $r_{\Rm}\geqs\frac{1}{10}r_i'$ on $P^*_-(x_i',t_i',A_ir_i')$.
		
		By $R^\phi\geqs-C$, Corollary \ref{nashosc} and Proposition \ref{Posc}, we have
		\begin{align*}
			\Npp_0(x_i',t_i')&\geqs\Npp_0(x_i,1)-CA_ir_i-\hn\log\left(\frac{t_i'}{1-(10A_ir_i)^2}\right),\\
			P^*_0(x_i',t_i')&\leqs P^*_0(x_1,1)-C(A_ir_i)^2\log(10A_ir_i).
		\end{align*}
		Then $\Nh_{x_i,1}(1)\to 0$, $A_ir_i\to 0$ implies $\Nh_{x_i',t_i'}(t_i')\to 0$. Note that $\int\phi_i(y,t_i')\dif\nu(y)=-\Np_{x_i',t_i'}(0)\to0$, $\frac{\dif}{\dif\tau}\int\phi_i\dif\nu=-\frac{1}{6}\int|H_i|^2\dif\nu\leqs0$. Then for $t\in[0,t_i']$, $\int\phi_i(t)\dif\nu$ uniformly converges to $0$.
		
		By \cite[Corollary 9.6]{bam1}, we can take a sequence $A_i'\leqs A_i$, $A_i'\to\infty$ such that $P_-(x_i',t_i',A_i'r_i')\subset P^*_-(x_i',t_i',A_ir_i')$, then $r_{\Rm}\geqs \frac{1}{10}r_i'$ on $P_-(x_i',t_i',A_i'r_i')$. Define $\tilde{g}_i(t):=r_i'^{-2}g_i(r_i'^2t+t_i'),$ $\tilde{H}_i(t):=r_i'^{-2}H_i(r_i'^2t+t_i')$, we obtain a sequence of flows $(\tilde{M}_i$,$(\tilde{g}_{i,t})_{t\in[-(A_i')^2,0]}$, $(\tilde{H}_{i,t}),(\tilde{\phi}_{i,t})))$ satisfies that $r_{\Rm}\geqs \frac{1}{10}$ on $P_-(x_i',0,A_i')$, $r_{\Rm}(x_i',0)=1$, $|\tilde{H}_{i,t}|^2\leqs Cr_i'^2\to 0$, and for any fixed $T>0$, $\tilde{\N}^{H}_{x_i',0}(T)=\Nh_{x_i',t_i'}(r_i'^2T)\to0$, $\int\tilde{\phi}_i(T)\dif\tilde{\nu}_i\to0$. 
		
		Then the injectivity radius at $(x_i',0)$ is uniformly bounded from below by Theorem \ref{noncollaped}. Since $|\Rm|$, $|H|$, $|\nabla H|$ are all uniformly bounded, after passing to a subsequence, these flows converge to a smooth and complete pointed ancient flow $(M_{\infty},(g_{\infty,t})_{t\leqs0},x_{\infty},H_{\infty}=0)$. Since $|\tilde{H}_i|^2\to0$, the limit flow is actually Ricci flow. By the upper bound of heat kernel and its gradient, $\tilde{K}_i(x_i',0;\cdot,\cdot)$ converge locally uniformly to a positive solution $v_\infty\in C^{\infty}(M_\infty\times\rean_-)$ of conjugate heat equation. Since $\int\tilde{\phi}_i(T)\dif\tilde{\nu}_i\to 0$, $\square\tilde{\phi}_i=\frac{1}{6}|\tilde{H}_i|^2\to0$,   $|\nabla \tilde{H}_i|^2\to 0$, we have $\tilde{\phi}_i\to \phi_{\infty}\equiv0$.
		
		Since for any fixed $T>0$, $\tilde{\N}^{H}_{x_i',0}(T)\to0$, then $\tilde{\W}^{H}_{x_i',0}(T)\geqs\int_T^{T+1}\tilde{\W}^{H}_{x_i',0}(s)\dif s=(T+1)\tilde{\N}^{H}_{x_i',0}(T+1)-T\tilde{\N}^{H}_{x_i',0}(T)\to 0$, which implies
		$$\int_{-T}^0\int_M\left(-2\tau\left|\ric^{\tilde{f}_i+\tilde{\phi}_i}-\frac{1}{2\tau}\tilde{g}_i\right|^2-\frac{1}{3\tau}\tilde{\Psi}_{x_i',0}\right)\dif\nu\dif t\to 0.$$
		Then $\ric^{f_\infty+\phi_{\infty}}-\frac{1}{2\tau}g_\infty=0$. $\phi_\infty=0$ implies $\ric+\nabla^2f_\infty-\frac{1}{2\tau}g_\infty=0$, thus $g_{\infty}$ is a gradient shrinking soliton. Since $|\Rm|$ is uniformly bounded, $g_\infty$ must be flat, which contradicts with $r_{\Rm}(x_\infty,0)=1<\infty$.
	\end{proof}
	\section{Some Estimates on $H$}
	In the previous sections, a bound on $|H|$ is always needed. Such bounds may not always hold without assumptions, so we will derive some estimates under centain curvature conditions.
	\subsection{Integral bounds on $H$}
	
	By the monotonicity of the weighted scalar curvature, (\ref{HL2phi}) gives the first integral bounds depends on the scalar curvature $R$ and the volume $\vol(g_t)$:
	\begin{equation*}
		\int_M|H|^2\dif g_t\leqs 12\int_M\left(R-R^\phi_0\right)\dif g_t.
	\end{equation*}
	\begin{proposition} We have
		\begin{equation*}
			\vol(g_t)\leqs e^{-3R^\phi_0t}\left(\vol(g_0)+\int_0^t\int_M2e^{3R^\phi_0s}R\dif g_s\dif s\right).
		\end{equation*}
		In particular, if $\int_M R\dif g_t\leqs C_0$, then we have
		\begin{itemize}
			\item If $R^\phi_0>0$, then $\vol(g_t)\leqs C$, $\int_M|H|^2\dif g_t\leqs C$.\\
			\item If $R^\phi_0=0$, then $\vol(g_t)\leqs C(t+1)$, $\int_M|H|^2\dif g_t\leqs C$.\\
			\item If $R^\phi_0<0$, then $\vol(g_t)\leqs Ce^{-3R^\phi_0t}$, $\int_M|H|^2\dif g_t\leqs Ce^{-3R^\phi_0t}$.
		\end{itemize}
		In particular, if $R\leqs C_0$, then we have
		\begin{equation*}
			\vol(g_t)\leqs Ce^{\left(2C_0-3R^\phi_0\right)t},\ \ \int_M|H|^2\dif g_t\leqs Ce^{Ct}.
		\end{equation*}
	\end{proposition}
	\begin{proof}
		\begin{equation*}
			\ddt\vol(g_t)=\int_M\left(\frac14|H|^2-R\right)\dif g_t\leqs2\int_M R\dif g_t-3R^\phi_0\vol(g_t).
		\end{equation*}
		Then the results follow from the Gronwall’s inequality.
	\end{proof}
	\begin{proposition}
		\begin{equation*}
			\ddt\int_M|H|^2\dif g_t=\int_M\left(-2|d^*H|^2-\frac32|H^2|^2+\frac14|H|^4+6\left<\ric,H^2\right>-R|H|^2\right)\dif g_t.
		\end{equation*}
	\end{proposition}
	\begin{proof}
		Direct calculation.
	\end{proof}
	\subsection{Pointwise bounds on $H$}
	In many cases, the integral bound is not enough. In previous sections, we need to estimate the term 
	\begin{equation*}
		\Psi_{x,t}(\tau)=\int_{t-\tau}^t\int_M|H|^2\dif\nu_{x,t;s}\dif s.
	\end{equation*}
	Note that $\Psi_{x,t}'(0)=\int_M|H|^2\dif\nu_{x,t;t}=|H|^2(x,t)$, then $\Psi_{x,t}(\tau)=|H|^2(x,t)\tau+O(\tau^2)$. Thus to estimate the term $\Psi_{x,t}$, the pointwise bound on $H$ is necessary.
	\begin{proposition}
		\begin{equation}\label{evolveH}
			\square|H|^2=-2|\nabla H|^2-\frac32|H^2|^2-6R_H\leqs-2|\nabla H|^2-\frac{3}{2n}|H|^4+C_n|\Rm||H|^2,
		\end{equation}
		where $R_H:=R_{ijkl}H^{ijr}H^{kls}g_{rs}$.
	\end{proposition}
	\begin{proof}
		\begin{equation*}
			\pdt|H|^2=3\left<2\ric-\frac12H^2,H^2\right>+2\left<\Delta_dH,H\right>,
		\end{equation*}
		where $\Delta_d=-(dd^*+d^*d)$ is Hodge Laplacian. Denote the connection Laplacian by $\Delta$, then by the Weitzenb\"{o}ck formula, we have
		\begin{equation*}
			2\left<\Delta_dH,H\right>=2\left<\Delta H-\ric(H),H\right>=\Delta|H|^2-2|\nabla H|^2-2\left<\ric(H),H\right>,
		\end{equation*}
		where $\ric(H)$ is the Weitzenb\"{o}ck curvature operator on the tensor defined by
		\begin{equation*}
			\ric(H)(X_1,\cdots,X_k):=\sum_{i,j}(R(e_j,X_i)H)(X_1,\cdots,e_j,\cdots,X_k).
		\end{equation*}
		In the normal coordinate,
		\begin{equation*}
			\begin{aligned}
				(\ric(H))_{ijk}=-\left(R_{sisr}H_{rjk}+R_{sijr}H_{srk}+R_{sikr}H_{sjr}+R_{sjir}H_{rsk}+R_{sjsr}H_{irk}+R_{sjkr}H_{isr}\right.\\
				\left.+R_{skir}H_{rjs}+R_{skjr}H_{irs}+R_{sksr}H_{ijr}\right),
			\end{aligned}
		\end{equation*}
		where
		\begin{equation*}
			-R_{sisr}H_{rjk}=R_{ir}H_{rjk},\ -R_{sjsr}H_{irk}=R_{jr}H_{irk},\ -R_{sksr}H_{ijr}=R_{kr}H_{ijr}.
		\end{equation*}
		By Bianchi identity,
		\begin{equation*}
			-R_{sijr}H_{srk}-R_{sjir}H_{rsk}=R_{ijsr}H_{srk}+R_{jsir}H_{srk}-R_{sjir}H_{rsk}=R_{ijsr}H_{srk}.
		\end{equation*}
		Thus
		\begin{equation*}
			(\ric(H))_{ijk}=R_{ir}H_{rjk}+R_{jr}H_{irk}+R_{kr}H_{ijr}+R_{ijsr}H_{srk}+R_{iksr}H_{sjr}+R_{jksr}H_{isr},
		\end{equation*}
		\begin{equation*}
			\left<\ric(H),H\right>=3\left<\ric,H^2\right>+3R_{ijsr}H_{srk}H_{ijk}=3\left<\ric,H^2\right>+3R_H,
		\end{equation*}
		\begin{equation*}
			\begin{aligned}
				\pdt|H|^2&=6\left<\ric,H^2\right>-\frac32|H^2|^2+\Delta|H|^2-2|\nabla H|^2-6\left<\ric,H^2\right>-6R_H\\
				&=\Delta|H|^2-2|\nabla H|^2-\frac32|H^2|^2-6R_H.
			\end{aligned}
		\end{equation*}
		The inequality of (\ref{evolveH}) follows from $|A|^2\geqs\frac{1}{n}(\tr A)^2$ and $|R_H|=|\Rm*H*H|\leqs C_n|\Rm||H|^2.$
	\end{proof}
	\begin{corollary}
		Suppose $|\Rm|\leqs K$, then $|H|^2\leqs \max\left(C_nK,\max_x |H|^2(x,0)\right)$.
	\end{corollary}
	\begin{proof}
		Apply maximum principle to (\ref{evolveH}).
	\end{proof}
	In general cases, bounds on Riemann curvature tensor seems to be necessary. However, the more important cases of generalized Ricci flow is in the lower dimension. For the dimensional reasons, the equation will be a bit simplified. Since the 3-form in 2-manifolds is trivial, the generalized Ricci flow is exactly the same as the Ricci flow. In dimension $n=3$, the tensor $H$ will be deduced to a scalar function:
	\begin{proposition}[\text{\cite[Proposition 4.38]{jeffGRF}}]
		Suppose $n=3$, $(g_t,H_t)$ solves the generalized Ricci flow, the function $h_t\in C^\infty(M)$ satisfies $H_t=h_t\dif g_t$, then
		\begin{equation*}
			\left\{\begin{aligned}
				\pdt g&=-2\ric+h^2g,\\
				\pdt h&=\Delta h+Rh-\frac32h^3.
			\end{aligned}
			\right.
		\end{equation*}
	\end{proposition}
	\begin{corollary}
		Suppose $n=3$, $|R|\leqs C_0$, then $|H|\leqs C$.
	\end{corollary}
	\begin{proof}
		Assume $h$ attains its maximum at $(x_0,t_0)$, $t_0>0$, $h(x_0,t_0)=\max_{t\in[0,T]}h(x,t)$.
		\begin{equation*}
			0\leqs\square h(x_0,t_0)\leqs Rh(x_0,t_0)-\frac32h^3(x_0,t_0).
		\end{equation*}
		Then $h(x_0,t_0)\leqs \left(\frac23C_0\right)^{\frac12}$. The lower bound is similar. The result follows from $|H|^2=6h^2$.
	\end{proof}
	
	In dimension $n=4$, the bounds on scalar curvature $R$ seems only to give a time-dependent bound on $||H||_{L^2}$. However, by a careful computation, we will show the Ricci curvature bound implies the pointwise bound on $H$.
	
	\begin{proposition}\label{H2coor}
		Suppose $n=4$, for a suitable orthonormal basis $\{e_i\}$ of $T_pM$, we have
		\begin{equation*}
			(H^2)_{ij}(p)=\frac13|H|^2\left(\begin{matrix}
				1 & & & \\
				&1& &  \\
				&	&1&	 \\
				&	& &0
			\end{matrix}\right).
		\end{equation*}
		Thus there exists a vector field $X_H$ as the eigenvector of $0$. In fact, we can take $X_H=\left(\star H\right)^\sharp\in\Gamma^\infty(TM)$.
	\end{proposition}
	\begin{proof}
		It's obvious if $H(p)=0$. So we assume $H(p)\neq0$. Take $\left(\star H\right)^\sharp=\lambda e_4$ and extend $e_4$ to an orthonormal basis $\{e_i\}$. Then
		\begin{align*}
			H=-\star(\lambda e^4)=\lambda e^1\wedge e^2\wedge e^3.
		\end{align*}
		So $\lambda^2=\frac16|H|^2$. $(H^2)_{4k}=(H^2)_{k4}=\left<i_{e_4}H,i_{e_k}H\right>=0$ for $k\in\{1,2,3,4\}$. For $i,j\in\{1,2,3\}$
		\begin{align*}
			(H^2)_{ij}=\sum_{k,l\in\{1,2,3\}}H_{ikl}H_{jkl}=\delta_{ij}\sum_{k,l\in\{1,2,3\}}H^2_{ikl}=2\lambda^2\delta_{ij}=\frac{1}{3}|H|^2\delta_{ij}.
		\end{align*}
	\end{proof}
	
	\begin{proposition}
		Suppose $n=4$, we have
		\begin{equation*}
			R_H=\frac13R|H|^2-2\left<\ric,H^2\right>=-\frac13R|H|^2+\frac23\ric(X_H)|H|^2.
		\end{equation*}
		Combine with (\ref{evolveH}), we have
		\begin{equation*}
			\square|H|^2=-2|\nabla H|^2-\frac12|H|^4+2R|H|^2-4\ric(X_H)|H|^2.
		\end{equation*}
		Thus $\ric\geqs-K$, $R\leqs K$ implies $|H|^2\leqs \max(12K,\max_x|H|^2(x,0))$.
	\end{proposition}
	\begin{proof}
		Take the normal coordinate and let $\ric=(R_{ij})$ to be diagonal.
		\begin{equation*}
			R_H=\sum_{i,j,k,l,r}R_{ijkl}H_{ijr}H_{klr}=\sum_{\{i,j\}=\{k,l\}}+\sum_{\{i,j\}\neq\{k,l\}},
		\end{equation*}
		\begin{equation*}
			\begin{aligned}
				\sum_{\{i,j\}=\{k,l\}}=-4\left[K_{12}(H^2_{123}+H^2_{124})+K_{13}(H^2_{123}+H^2_{134})+K_{14}(H^2_{124}+H^2_{134})\right.\\
				\left.+K_{23}(H^2_{123}+H^2_{234})+K_{24}(H^2_{124}+H^2_{234})+K_{34}(H^2_{134}+H^2_{234})\right],
			\end{aligned}
		\end{equation*}
		where $K_{ij}=R_{ijji}$ is the sectional curvature.
		Note that
		\begin{equation*}
			(K_{12}+K_{13}+K_{23})H^2_{123}=\left(\frac12R-R_{44}\right)H^2_{123}.
		\end{equation*}
		Thus we have
		\begin{equation*}
			\sum_{\{i,j\}=\{k,l\}}=-2R\left(H^2_{123}+H^2_{124}+H^2_{134}+H^2_{234}\right)+4\left(R_{11}H^2_{234}+R_{22}H^2_{134}+R_{33}H^2_{124}+R_{44}H^2_{123}\right).
		\end{equation*}
		Using
		\begin{equation*}
			(H^2)_{11}=2(H^2_{123}+H^2_{124}+H^2_{134})=2\left(\frac16|H|^2-H^2_{234}\right),
		\end{equation*}
		we have
		\begin{equation}\label{Rh4ij=kl}
			\begin{aligned}
				\sum_{\{i,j\}=\{k,l\}}&=-\frac13R|H|^2+\frac23|H|^2\sum_i R_{ii}-2\sum_i R_{ii}(H^2)_{ii}\\
				&=\frac13 R|H|^2-2\left<\ric,H^2\right>.
			\end{aligned}
		\end{equation}
		\begin{equation}\label{Rh4ijn=kl}
			\begin{aligned}
				\sum_{\{i,j\}\neq\{k,l\}}=8&\left[(R_{1213}+R_{2434})H_{124}H_{134}+(R_{1214}+R_{2334})H_{123}H_{143}+(R_{1223}+R_{1434})H_{124}H_{234}\right.\\
				&\left.(R_{1224}+R_{1334})H_{123}H_{243}+(R_{1314}+R_{2324})H_{132}H_{142}+(R_{1323}+R_{1424})H_{134}H_{234}\right]\\
				=8&\left(-R_{23}H_{124}H_{134}-R_{24}H_{123}H_{143}+R_{13}H_{124}H_{234}+R_{14}H_{123}H_{243}\right.\\
				&\ \ \left.-R_{34}H_{132}H_{142}-R_{12}H_{134}H_{234}\right)\\
				=0&.
			\end{aligned}
		\end{equation}
		Combine (\ref{Rh4ij=kl}) and (\ref{Rh4ijn=kl}), we have
		\begin{equation*}
			R_H=\frac13R|H|^2-2\left<\ric,H^2\right>.
		\end{equation*}
		For the suitable orthonormal basis in \Cref{H2coor}, $X_H=e_4$.
		\begin{equation*}
			\left<\ric,H^2\right>=\sum_i R_{ii}(H^2)_{ii}=\frac13|H|^2\left(R_{11}+R_{22}+R_{33}\right)=\frac13R|H|^2-\frac13R_{44}|H|^2,
		\end{equation*}
		\begin{equation*}
			R_H=-\frac13R|H|^2+\frac23\ric(X_H)|H|^2.
		\end{equation*}
	\end{proof}
	Recently, an interesting bound of torsion is found in the case of the pluriclosed flow on complex surface with special initial data, see \cite{ye}.
	\begin{theorem}[\text{\cite[Theorem 1.3]{ye}}]
		For a compact complex surface $(M^4,J,g(t))$ that admits a pluriclosed flow with Hermitian-symplectic initial data on $[0,\tau)$, if the Chern scalar curvature satisfies
		\begin{align*}
			R^C(x,t)\leqslant K,\  x\in M, t\in [0,\tau),
		\end{align*} 
		then we have
		\begin{align*}
			|H(x,t)|^2_{g(x,t)}\leqslant\max\left\{2K,\max_x|H(x,0)|^2_{g(x,0)} \right\},\ x\in M, t\in [0,\tau).
		\end{align*}
	\end{theorem}
	
	\bibliographystyle{ytamsalpha}
	\bibliography{ref}
	
\end{document}